\documentclass[reqno,11pt]{amsart}
\usepackage{graphicx}
\usepackage{amsmath,amsthm,amsfonts,amssymb,latexsym,epsfig}
\newtheorem{theorem}{Theorem}[section]
\newtheorem{lemma}[theorem]{Lemma}

\theoremstyle{definition}

\usepackage{hyperref}
\theoremstyle{remark}

\usepackage{orcidlink}
\usepackage{pdflscape}
\numberwithin{equation}{section}
\usepackage{diagbox}
\usepackage{booktabs} 
\usepackage{array}
\usepackage[a4paper, left=1in, right=1in, top=1in, bottom=1in]{geometry}

\numberwithin{equation}{section}

\begin{document}

\title[short text for running head]{``2-color partitions modulo powers of $7$''}
\title{On 2-color partitions where one of the colors is multiples of $7^k$ }

%    author one information
\author[]{D. S. Gireesh$^{1}$\orcidlink{0000-0002-2804-6479}}
\address{$^{3}$Department of Mathematics, BMS College of Engineering, P.O. Box No.: 1908, Bull Temple Road,
Bengaluru-560 019, Karnataka, India.}
%\curraddr{}
\email{gireeshdap@gmail.com}
\thanks{}

%    author two information
\author[]{Shivashankar C.$^{2}$\orcidlink{0000-0001-7503-4414}}
\address{$^{1}$Department of Mathematics, Vidyavardhaka College of Engineering, Gokulam III Stage, Mysuru-570002, Karnataka, India.}
%\curraddr{}
\email{shankars224@gmail.com}
\thanks{}

%    author three information
\author[]{HemanthKumar B.$^{3}$\orcidlink{0000-0001-7904-293X}}
\address{$^2$Department of Mathematics, RV College of Engineering, RV Vidyanikethan Post, Mysore Road, Bengaluru-560 059, Karnataka, India.}
\email{hemanthkumarb.30@gmail.com}

\thanks{}

\date{}
          
\begin{abstract}
In this work, we investigate the arithmetic properties of $p_{1,7^k}(n)$, which counts 2-color partitions of $n$ where one of the colors appears only in parts that are multiples of $7^k$. By constructing generating functions for $p_{1,7^k}(n)$ across specific arithmetic progressions, we establish a set of Ramanujan-type infinite family of congruences modulo powers of $7$.
\end{abstract}

\medskip

\maketitle

\section{Introduction}
\label{intro}
A partition of a positive integer $n$ is a non-increasing sequence of positive integers whose sum equals $n$. Let $p(n)$ denote the number of partitions of $n$, with the generating function given by
\[\sum\limits_{n\geq0}p(n)q^n=\frac1{f_1}.\]
Throughout this paper, we define
\[f_r:=(q^r;q^r)_\infty=\prod\limits_{m=1}^{\infty}(1-q^{rm}).\]

Ramanujan \cite{Ram2} conjectured that if $\ell\in\{5,7,11\}$ and $0<\delta_{\ell,\beta}<\ell^\beta$ satisfying $24\delta_{\ell,\beta}\equiv 1\pmod{\ell^\beta}$, the partition function satisfies the congruence
\begin{equation}\label{RC}
    p\left(\ell^\beta n+\delta_{\ell,\beta}\right)\equiv 0\pmod{\ell^\beta}.
\end{equation}
Watson \cite{GNW} later proved this conjecture for $\ell=5$. However, Chowla \cite{CS} observed that since $p(243)$ is not divisible by $7^3$, this contradicts the conjecture for $\ell=7$ and $\beta=3$. The conjecture was subsequently refined and proved by Watson \cite{GNW}, establishing that for $k\geq1$,
\begin{equation}\label{pc11}
    p(7^{2k-1}n+\lambda_{2k-1})\equiv 0\pmod{7^k},
\end{equation}
and
\begin{equation}\label{pc12}
    p(7^{2k}n+\lambda_{2k})\equiv 0\pmod{7^{k+1}},
\end{equation}
where $\lambda_k$ is the modular inverse of $24$ modulo $7^k$.

Using the classical identities of Euler and Jacobi, Hirschhorn and Hunt \cite{HH} provided a simple proof of \eqref{RC} for $\ell=5$. Later, Garvan \cite{FG} extended these results to $\ell=7$ by deriving the generating functions
\begin{equation}\label{H1}
\sum_{n\geq 0}p\left(7^{2k-1}n+\lambda_{2k-1}\right)q^n=\sum_{j\geq 1}x_{2k-1,j} \, q^{j-1}\frac{f_7^{4j-1}}{f_1^{4j}}
\end{equation}
and
\begin{equation}\label{H2}
\sum_{n\geq 0}p\left(7^{2k}n+\lambda_{2k}\right)q^n=\sum_{j\geq 1}x_{2k,j} \, q^{j-1}\frac{f_7^{4j}}{f_1^{4j+1}},
\end{equation}
where the coefficient vectors $\mathbf{x}_k= (x_{k,1}, x_{k,2},\dots)$ are defined recursively by
$$\mathbf{x}_1=\left(x_{1,1},x_{1,2},\dots\right)=\left(7,7^2,0,\dots\right),$$ and for $k\geq1$,
\[
   x_{k+1, i}=\begin{cases}
  \displaystyle \sum_{j\geq1} x_{k, j} \,m_{4j,j+i}, \,\,\text{if}\,\,k\,\,\text{is odd},\\
   \displaystyle \sum_{j\geq1}x_{k, j} \,m_{4j+1,j+i}, \,\,\text{if}\, \,k \,\,\text{is even},
    \end{cases}
\]
where the first seven rows of $M=\left(m_{i,j}\right)_{i,j\geq 1}$ are
\begin{table}[h!]
\centering
\caption{Values of $m_{i,j}$, where $1 \leq i \leq 7$ and $1 \leq j \leq 7$.}
\renewcommand{\arraystretch}{1.5} % Adjust row height for readability
\setlength{\tabcolsep}{4pt} % Adjust column spacing
\begin{tabular}{|c|ccccccc|}
\hline
\diagbox[width=2em]{$i$}{$j$} & $1$ & $2$ & $3$ & $4$ & $5$ & $6$ & $7$ \\ \hline
1 & $7$ & $7^2$ & $0$ & $0$ & $0$ & $0$ & $0$ \\
2 & $10$ & $9 \times 7^2$ & $2 \times 7^4$ & $7^5$ & $0$ & $0$ & $0$ \\
3 & $3$ & $114 \times 7$ & $85 \times 7^3$ & $24 \times 7^5$ & $3 \times 7^7$ & $7^8$ & $0$ \\
4 & $0$ & $82 \times 7$ & $176 \times 7^3$ & $845 \times 7^4$ & $272 \times 7^6$ & $46 \times 7^8$ & $4 \times 7^{10}$ \\
5 & $0$ & $190$ & $1265 \times 7^2$ & $1895 \times 7^4$ & $1233 \times 7^6$ & $3025 \times 7^7$ & $620 \times 7^9$ \\
6 & $0$ & $27$ & $736 \times 7^2$ & $16782 \times 7^3$ & $20424 \times 7^5$ & $12825 \times 7^7$ & $4770 \times 7^9$ \\
7 & $0$ & $1$ & $253 \times 7^2$ & $1902 \times 7^4$ & $4246 \times 7^6$ & $31540 \times 7^7$ & $19302 \times 7^9$ \\ \hline
\end{tabular}
\end{table}

\begin{table}[h!]
\centering
\caption{Values of $m_{i,j}$, where $1 \leq i \leq 7$ and $8 \leq j \leq 14$.}
\renewcommand{\arraystretch}{1.5} % Adjust row height for readability
\setlength{\tabcolsep}{4pt} % Adjust column spacing
\begin{tabular}{|c|ccccccc|}
\hline
\diagbox[width=2em]{$i$}{$j$} & $8$ & $9$ & $10$ & $11$ & $12$ & $13$ & $14$ \\ \hline
1 & $0$ & $0$ & $0$ & $0$ & $0$ & $0$ & $0$ \\
2 & $0$ & $0$ & $0$ & $0$ & $0$ & $0$ & $0$ \\
3 & $0$ & $0$ & $0$ & $0$ & $0$ & $0$ & $0$ \\
4 & $7^{11}$ & $0$ & $0$ & $0$ & $0$ & $0$ & $0$ \\
5 & $75 \times 7^{11}$ & $5 \times 7^{13}$ & $7^{14}$ & $0$ & $0$ & $0$ & $0$ \\
6 & $7830 \times 7^{10}$ & $1178 \times 7^{12}$ & $111 \times 7^{14}$ & $6 \times 7^{16}$ & $7^{17}$ & $0$ & $0$ \\
7 & $7501 \times 7^{11}$ & $1944 \times 7^{13}$ & $2397 \times 7^{14}$ & $285 \times 7^{16}$ & $22 \times 7^{18}$ & $7^{20}$ & $7^{20}$ \\ \hline
\end{tabular}
\end{table}
and for $i\geq4$, $m_{i,1}=0$, and for $i\geq 8$, $m_{i,2}=0$ and for $i\geq8$, $j\geq3$,
\begin{align}
\nonumber m_{i,j}&=7m_{i-3,j-1}+35m_{i-2,j-1}+49m_{i-1,j-1}+m_{i-7,j-2}+7m_{i-6,j-2}\\&+21m_{i-5,j-2}+49m_{i-4,j-2}+147m_{i-3,j-2}+343m_{i-2,j-2}+343m_{i-1,j-2}.\label{mij}
\end{align}

Let $p_{1,\ell}(n)$ be the number of 2-color partitions of $n$ where one of the colors appears only in parts that are multiples of $\ell$; its generating function is given by
\begin{equation}\label{bl}
\sum\limits_{n\geq0}p_{1,\ell}(n)q^n=\frac{1}{f_1f_{\ell}}.
\end{equation}

Ahmed, Baruah, and Dastidar \cite{ABD} discovered new congruences modulo $5$, proving that for certain values of $t$
\begin{equation}\label{AB}
p_{1,\ell}\left(25n+t\right)\equiv 0 \pmod{5},
\end{equation}
where $\ell\in\{0, 1, 2, 3, 4, 5, 10, 15, 20\}$ and $\ell+t=24$. They further conjectured similar congruences for $\ell= 7, 8, 17$ which were confirmed by Chern \cite{SC}. Additionally, Chern \cite{SC} established a new family of congruences:
\[p_{1,4}\left(49n+t\right)\equiv 0 \pmod{7}, \text{ for } t\in\{11, 25, 32, 39\}.\]

Further generalizations were made by Wang \cite{LW}, who derived congruences for $p_{1,5}(n)$, strengthening prior results. Ranganath \cite{RD} extended these findings to $p_{1,25}(n)$ and $p_{1,49}(n)$, proving that for all $n,\beta \geq0$,
\begin{equation}\label{RD1}
p_{1,25}\left(5^{2\beta+1}n+\frac{7\times5^{2\beta+1}+13}{12}\right)\equiv0\pmod{5^{\beta+1}},
\end{equation}
\begin{equation}\label{RD2}
p_{1,25}\left(5^{2\beta+2}n+\frac{11\times5^{2\beta+2}+13}{12}\right)\equiv0\pmod{5^{\beta+2}},
\end{equation}
\begin{equation}\label{RD3}
p_{1,49}\left(7^{2\beta+1}n+\frac{5\times7^{2\beta+1}+25}{12}\right)\equiv0\pmod{7^{\beta+1}},
\end{equation}
and
\begin{equation}\label{RD4}
p_{1,49}\left(7^{2\beta+2}n+\frac{11\times7^{2\beta+2}+25}{12}\right)\equiv0\pmod{7^{\beta+2}}.
\end{equation}
In this paper, we establish congruences for $p_{1,7^k}$ for all positive integers $k$. The results presented here are generalized and extend those of Ranganath \cite{RD}.

The main results are as follow:
\begin{theorem}\label{th1}
For each $n, \beta \geq0$, and $k\geq1$, we have
\begin{equation}\label{c1}
p_{1,7^{2k-1}}\left(7^{2k+\beta-1}n+\frac{16\cdot7^{2k+\beta-1}+7^{2k-1}+1}{24}\right)\equiv 0\pmod{ 7^{k}},
\end{equation}
\begin{equation}\label{c2}
p_{1,7^{2k}}\left(7^{2k+2\beta-1}n+\frac{10\cdot7^{2k+2\beta-1}+7^{2k}+1}{24}\right)\equiv 0\pmod{ 7^{k+\beta}},
\end{equation}
\begin{equation}\label{c3}
p_{1,7^{2k}}\left(7^{2k+2\beta}n+\frac{22\cdot7^{2k+2\beta}+7^{2k}+1}{24}\right)\equiv 0\pmod{ 7^{k+\beta+1}},
\end{equation}
\begin{equation}\label{c4}
p_{1,7^{2k}}\left(7^{2k+2\beta+1}n+\frac{(24r+22)\cdot7^{2k+2\beta}+7^{2k}+1}{24}\right)\equiv 0\pmod{ 7^{k+\beta+2}},
\end{equation}
where $r\in \{3,4,6\}.$
\end{theorem}

\section{Preliminaries}
In this section, we state some lemmas which play a vital role in proving our results.
Let $H$ be the “huffing” operator modulo $7$, that is,
\[H\left(\sum{p_nq^n}\right)=\sum{p_{1,7n}q^{7n}}.\]
\begin{lemma}[\cite{FG}, Lemma 3.1 and 3.4]
 If $\xi= \frac{f_1}{q^2 f_{49}}$ and $T = \frac{f_7^4}{q^7 f_{49}^4}$, then
\begin{equation}\label{Gu1}
  H \left( \xi^{-i} \right)  = \sum_{j\geq1} m_{i,j} T^{-j}
\end{equation}
and
\begin{equation}\label{Gu2}
  H \left( \xi^{-4i} \right)  = \sum_{j\geq1} m_{4i,i+j} T^{-i-j}. 
\end{equation}
\end{lemma}

\begin{lemma}\label{Hi1}
For all $i\geq1$, we have

\begin{equation}\label{H4i+1}
H \left( q^{i+2} \frac{f_{7}^{4i}}{f_1^{4i+1}} \right)= \sum_{j\ge 1} m_{4i+1,i+j} q^{7j} \frac{f_{49}^{4j-1}}{f_7^{4j}},
\end{equation}
\begin{equation}\label{H4i+2}
H \left( q^{i+4} \frac{f_{7}^{4i}}{f_1^{4i+2}} \right)= \sum_{j \ge 1}  m_{4i+2	,i+j} \, q^{7j} \frac{f_{49}^{4j-2}}{f_7^{4j}},
\end{equation}
and
\begin{equation}\label{H4i}
H \left( q^{i} \frac{f_{7}^{4i}}{f_1^{4i}} \right)= \sum_{j\ge 1} m_{4i, i+j} \, q^{7j} \frac{f_{49}^{4j}}{f_7^{4j}}.
\end{equation}
%\begin{equation}\label{H6i+1}
%H \left( q^{i+1} \frac{f_{5}^{6i+1}}{f_1^{6i+1}} \right) = \sum_{j=1}^{5j+1} m_{6i+1, i+j} \,q^{5j} \frac{f_{25}^{6j-1}}{f_5^{6j-1}}.
%\end{equation}
\end{lemma}
\begin{proof}
We can rewrite \eqref{Gu1} as
\begin{equation}\label{Hi}
H \left( q^{2i} \frac{f_{49}^i}{f_1^i} \right) = \sum_{j\ge 1}m_{i,j} \, q^{7j} \frac{f_{49}^{4j}}{f_7^{4j}}.
\end{equation}
From \eqref{Hi} and the fact that $m_{4i+1, j}=0$ for $1\leq j \leq i,$ we have
\begin{align*}
H\left(\left(q^2\frac{f_{49}}{f_1}\right)^{4i+1}\right)&=\sum_{j=i+1}^{4i+1} m_{4i+1,j} \,q^{7j} \frac{f_{49}^{4j}}{f_{7}^{4j}}.\\
&=\sum_{j\ge 1} m_{4i+1,i+j} \,q^{7i+7j} \frac{f_{49}^{4i+4j}}{f_{7}^{4i+4j}},
\end{align*}
which yields \eqref{H4i+1}. Similarly, we can prove \eqref{H4i+2} and \eqref{H4i}.
\end{proof}

\section{Generating functions}
In this section, we establish generating functions for $p_{1,7^{\ell}}(n)$ within specific arithmetic progressions.
\begin{theorem}\label{T1}
For each $\beta \geq0$ and $k\geq1$, we have
\begin{equation}\label{G1}
    \sum_{n\geq 0}p_{1,7^{2k-1}}\left(7^{2k+\beta-1}n+\frac{16\cdot7^{2k+\beta-1}+7^{2k-1}+1}{24}\right)q^n=\sum_{i\geq 1}y^{(2k-1)}_{\beta+1,i} \, q^{i-1}\frac{f_7^{4i-1}}{f_1^{4i+1}},
\end{equation}
%and 
%\begin{equation}\label{G2}
%\sum_{n\geq 0}p_{1,5^{2k-1}}\left(5^{2k+2\beta}n+\frac{4\cdot5^{2k+2\beta}-5^{2k-1}+1}{24}\right)q^n=\sum_{i\geq 1}y^{(2k-1)}_{2\beta+2,i} \, q^{i-1}\frac{f_5^{6i-5}}{f_1^{6i-5}}.
%\end{equation}
where the coefficient vectors are defined as follows:
\begin{equation*}
y_{1,j}^{(2k-1)}=x_{2k-1,j} \text{ and }
y^{(2k-1)}_{\beta+1,j}= \sum_{i\geq1}y^{(2k-1)}_{\beta,i} m_{4i+1,j+i}
\end{equation*}
for all $\beta, j\geq 1$.
\end{theorem}
\begin{proof}
From \eqref{bl}, we have
\begin{equation}\label{p1}
    \sum\limits_{n\geq0}p_{1,7^{2k-1}}(n)q^n=\dfrac{1}{f_{7^{2k-1}}}\sum\limits_{n\geq0}p(n)q^n.
\end{equation}
Extracting the terms involving $q^{7^{2k-1}n+\lambda_{2k-1}}$ on both sides of \eqref{p1} and dividing throughout by $q^{\lambda_{2k-1}}$, we obtain  
\begin{equation*}
    \sum\limits_{n\geq0}p_{1,7^{2k-1}}(7^{2k-1}n+\lambda_{2k-1})q^{7^{2k-1}n}=\dfrac{1}{f_{7^{2k-1}}}\sum\limits_{n\geq0}p(7^{2k-1}n+\lambda_{2k-1})q^{7^{2k-1}n}.
\end{equation*}
If we replace $q^{7^{2k-1}}$ by $q$ and  use \eqref{H1}, we get
\begin{equation*}
    \sum\limits_{n\geq0}p_{1,7^{2k-1}}(7^{2k-1}n+\lambda_{2k-1})q^n=\sum_{j\geq 1}x_{2k-1,j} \, q^{j-1}\frac{f_7^{4j-1}}{f_1^{4j+1}},
\end{equation*}
which is the case $\beta=0$ of \eqref{G1}.

We now assume that \eqref{G1} is true for some integer $\beta\geq 0$. Applying the operator $H$ to both sides, by \eqref{H4i+1}, we have

\begin{equation*}
\begin{split}
    \sum_{n\geq 0}p_{1,7^{2k-1}}\left(7^{2k+\beta}n+\frac{16\cdot7^{2k+\beta}+ 7^{2k-1}+1}{24}\right)q^{7n+7}&=\sum_{i\geq 1}y^{(2k-1)}_{\beta+1,i} \, H\left(q^{i+2}\frac{f_7^{4i}}{f_1^{4i+1}}\right)\times \frac{1}{f_7}
    \\&=\sum_{i\geq 1}y^{(2k-1)}_{\beta+1,i}\sum_{j=1}^{4i+1} m_{4i+1,i+j}\, q^{7j} \frac{f_{49}^{4j-1}}{f_7^{4j+1}}\\&
    =\sum_{j\geq 1}\left(\sum_{i\geq1}y^{(2k-1)}_{\beta+1,i} \, m_{4i+1,i+j}\right) q^{7j} \frac{f_{49}^{4j-1}}{f_7^{4j+1}}\\&
    =\sum_{j\geq 1}y^{(2k-1)}_{\beta+2,j} \,q^{7j} \frac{f_{49}^{4j-1}}{f_7^{4j+1}}.
    \end{split}
\end{equation*}
That is,
\begin{equation*}
    \sum_{n\geq 0}p_{1,7^{2k-1}}\left(7^{2k+\beta}n+\frac{16\cdot7^{2k+\beta}+ 7^{2k-1}+1}{24}\right)q^{n}=\sum_{i\geq 1}y^{(2k-1)}_{\beta+2,i} \, q^{i-1} \frac{f_{7}^{4i-1}}{f_1^{4i+1}}.
\end{equation*}
So we obtain \eqref{G1} with $\beta$ replaced by $\beta+1$.

%Suppose that \eqref{G2} is true for some integer $\beta \geq 0$. Applying the operating $H$ to both sides, by \eqref{H6i-5}, we obtain
%\begin{equation*}
%\begin{split}
%       & \sum_{n\geq 0}b_{5^{2k-1}}\left(5^{2k+2\beta+1}n+\frac{4\cdot5^{2k+2\beta+2}-5^{2k-1}+1}{24}\right)q^{5n+5}\\&=\sum_{i\geq 1}y^{(2k-1)}_{2\beta+2,i} \sum_{j=1}^{5i-4} m_{6i-5,i+j-1} q^{5j} \, \frac{f_{25}^{6j-1}}{f_5^{6j-1}}\\&
%       =\sum_{j\geq 1}\left(\sum_{i\geq1}y^{(2k-1)}_{2\beta+2,i} \, m_{6i-5,i+j-1}\right) q^{5j} \frac{f_{25}^{6j-1}}{f_5^{6j-1}}\\&
%       =\sum_{j\geq 1}y^{(2k-1)}_{2\beta+3,j} \, q^{5j} \frac{f_{25}^{6j-1}}{f_5^{6j-1}},
%\end{split}
%\end{equation*}
%which yields
%\begin{equation*}
%    \sum_{n\geq 0}b_{5^{2k-1}}\left(5^{2k+2\beta+1}n+\frac{4\cdot5^{2k+2\beta+2}-5^{2k-1}+1}{24}\right)q^{n}=\sum_{j\geq 1}y^{(2k-1)}_{2\beta+3,j} \, q^{j-1} \frac{f_{5}^{6j-1}}{f_1^{6j-1}}.
%\end{equation*}
%This is \eqref{G1} with $\beta$ replaced by $\beta+1$. This completes the proof.
\end{proof}

%%%%%%%%%%%%%%%%%%
\begin{theorem} \label{T2}
For each $\beta \geq0$, and $k\geq1$, we have
\begin{equation}\label{G3}
    \sum_{n\geq 0}p_{1,7^{2k}}\left(7^{2k+2\beta-1}n+\frac{10\cdot7^{2k+2\beta-1}+7^{2k}+1}{24}\right)q^n=\sum_{i\geq 1}y^{(2k)}_{2\beta+1,i}\,q^{i-1}\frac{f_7^{4i-2}}{f_1^{4i}},
\end{equation}
and
\begin{equation}\label{G4}
    \sum_{n\geq 0}p_{1,7^{2k}}\left(7^{2k+2\beta}n+\frac{22\cdot7^{2k+2\beta}+7^{2k}+1}{24}\right)q^{n}=\sum_{j\geq 1}y^{(2k)}_{2\beta+2,j}q^{j-1} \frac{f_{7}^{4j}}{f_1^{4j+2}},
\end{equation}
where coefficient vectors are defined as follow:
\begin{equation*}
        y_{1,j}^{(2k)}=x_{2k-1,j}
\end{equation*}
and
\begin{equation*}
	y^{(2k)}_{\beta+1,j}= 
	\begin{cases}
		\displaystyle \sum_{i\geq1}y^{(2k)}_{\beta,i} m_{4i,j+i} \,\,\   & \text{if  $\beta$ is odd},\\
		\displaystyle \sum_{i\geq1}y^{(2k)}_{\beta,i}m_{4i+2,j+i} \,\,\   & \text{if  $\beta$ is even},
	\end{cases}
\end{equation*}
for all $\beta, j\geq 1$.
\end{theorem}
\begin{proof}
    Equation \eqref{bl} can be written as
\begin{equation*}
    \sum\limits_{n\geq0}p_{1,7^{2k}}(n)q^n=\dfrac{1}{f_{7^{2k}}}\sum\limits_{n\geq0}p(n)q^n.
\end{equation*}
Extracting the terms involving $q^{7^{2k-1}n+\lambda_{2k-1}}$ on both sides of the above equation and dividing throughout by $q^{\lambda_{2k-1}}$, we obtain  
\begin{equation}\label{p6}
    \sum\limits_{n\geq0}p_{1,7^{2k}}(7^{2k-1}n+\lambda_{2k-1})q^{7^{2k-1}n}=\dfrac{1}{f_{7^{2k}}}\sum\limits_{n\geq0}p(7^{2k-1}n+\lambda_{2k-1})q^{7^{2k-1}n}.
\end{equation}
If we replace $q^{7^{2k-1}}$ by $q$ and use \eqref{H1}, we get
\begin{equation}\label{p7}
    \sum\limits_{n\geq0}p_{1,7^{2k}}(7^{2k-1}n+\lambda_{2k-1})q^n=\sum_{j\geq 1}x_{2k-1,j} \, q^{j-1}\frac{f_7^{4j-2}}{f_1^{4j}},
\end{equation}
which is the case $\beta=0$ of \eqref{G3}.

We now assume that \eqref{G3} is true for some integer $\beta\geq0$. Applying the operator $H$ to both sides, by \eqref{H4i}, we have
\begin{equation*}
    \begin{split}
      \sum_{n\geq 0}p_{1,7^{2k}}\left(7^{2k+2\beta}n+\frac{22\cdot7^{2k+2\beta}+7^{2k}+1}{24}\right)q^{7n+7}&=\sum_{i\geq 1}y^{(2k)}_{2\beta+1,i} \, H\left(q^{i}\frac{f_7^{4i}}{f_1^{4i}}\right)\times\frac{1}{f_7^2}\\&=\sum_{i\geq 1}y^{(2k)}_{2\beta+1,i}\sum_{j\ge1} m_{4i, i+j} \,q^{7j} \frac{f_{49}^{4j}}{f_7^{4j+2}}\\&
        =\sum_{j\geq 1}\left(\sum_{i\geq1}y^{(2k)}_{2\beta+1,i} \, m_{4i, i+j} \right)q^{7j} \frac{f_{49}^{4j}}{f_7^{4j+2}}\\&
        =\sum_{j\geq 1}y^{(2k)}_{2\beta+2,j}\, q^{7j} \frac{f_{49}^{4j}}{f_7^{4j+2}},
    \end{split}
\end{equation*}
That is
\begin{equation*}
    \sum_{n\geq 0}p_{1,7^{2k}}\left(7^{2k+2\beta}n+\frac{22\cdot7^{2k+2\beta}+7^{2k}+1}{24}\right)q^{n}=\sum_{j\geq 1}y^{(2k)}_{2\beta+2,j}q^{j-1} \frac{f_{7}^{4j}}{f_1^{4j+2}}.
\end{equation*}
Hence, if \eqref{G3} is true for some integer $\beta \geq 0$, then \eqref{G4} is true for $\beta$.

Suppose that \eqref{G4} is true for some integer $\beta \geq 0$. Applying the operating $H$ to both sides, by \eqref{H4i+2}, we obtain
\begin{equation*}
\begin{split}
       \sum_{n\geq 0}p_{1,7^{2k}}\left(7^{2k+2\beta+1}n+\frac{10\cdot7^{2k+2\beta+1}+7^{2k}+1}{24}\right)q^{7n+7}&=\sum_{i\geq 1}y^{(2k)}_{2\beta+2,i} \sum_{j\ge 1} m_{4i+2,i+j} q^{7j} \, \frac{f_{49}^{4j-2}}{f_7^{4j}}\\&
       =\sum_{j\geq 1}\left(\sum_{i\geq1}y^{(2k)}_{2\beta+2,i} \, m_{4i+2,i+j}\right) q^{7j} \frac{f_{49}^{4j-2}}{f_7^{4j}}\\&
       =\sum_{j\geq 1}y^{(2k)}_{2\beta+3,j} \, q^{7j} \frac{f_{49}^{4j-2}}{f_7^{4j}},
\end{split}
\end{equation*}
which yields
\begin{equation*}
    \sum_{n\geq 0}p_{1,7^{2k}}\left(7^{2k+2\beta+1}n+\frac{10\cdot7^{2k+2\beta+1}+7^{2k}+1}{24}\right)q^{n}=\sum_{j\geq 1}y^{(2k)}_{2\beta+3,j} \, q^{j-1} \frac{f_{7}^{4j-2}}{f_1^{4j}}.
\end{equation*}
This is \eqref{G3} with $\beta$ replaced by $\beta+1$. This completes the proof.
\end{proof}

%%%%%%%%%%%%%%%%%%

\section{Proof of Congruences}

For a positive integer $n$, let $\pi(n)$ be the highest power of $7$ that divides $n$, and define $\pi(0)=+\infty$.
\begin{lemma}[\cite{FG}, Lemma 5.1]
    For each $i,j\geq1 $, we have \begin{equation}\label{pmij} \pi\left(m_{i,j}\right)\geq\left[\frac{7j-2i-1}{4}\right].
    \end{equation}
     \end{lemma}
In particular,
\begin{align}
\pi\left(m_{4i,i+j}\right)&\geq\left[\frac{7(i+j)-2(4i)-1}{4}\right]=\left[\frac{7j-i-1}{4}\right],\label{p4i-1} \\
\pi\left(m_{4i+1,i+j}\right)&\geq\left[\frac{7(i+j)-2(4i+1)-1}{4}\right]=\left[\frac{7j-i-3}{4}\right],\label{p4i-3} 
\end{align}
and
\begin{equation}\label{p4i} 
\pi\left(m_{4i+2,i+j}\right)\geq\left[\frac{7(i+j)-2(4i+2)-1}{4}\right]=\left[\frac{7j-i-5}{4}\right].
\end{equation}
\begin{lemma}[\cite{FG}, Lemma 5.3]
    For all $k,j\geq1 $, we have
    \begin{equation}\label{pi1}
  \pi(x_{1,1})=1,\,\,\, \pi(x_{1,2}) =2,\,\,\, \pi\left(x_{2k+1,j}\right)\geq k+1+\left[\frac{7j-4}{4}\right]
     \end{equation} 
    and
    \begin{equation}\label{pi11}
   \pi\left(x_{2k,j}\right)\geq k+1+\left[\frac{7j-6}{4}\right],
     \end{equation}
where \begin{equation*}
   \delta_{j, 1}=\begin{cases}
   1 & \text{if $j=1$},\\
    0 & \text{if $j\neq 1$}.
    \end{cases}
\end{equation*}
    \end{lemma}
\begin{lemma}
    For all $j,k\geq1$ and $\beta\geq 0$, we have
\begin{equation}\label{pi4+1}
   \pi\left(y^{(2k-1)}_{\beta+1,j}\right)\geq k+\left[\frac{7j-4}{4}\right].
\end{equation}
\end{lemma}
\begin{proof}
 In view of Theorem \ref{T1}, we have
\begin{equation*}
        y_{1,j}^{(2k-1)}=x_{2k-1,j}.
\end{equation*}
From \eqref{pi1}, we can see that 
\begin{equation*}
\pi\left(x_{2k-1,j}\right)\geq k+\left[\frac{7j-4}{4}\right],
\end{equation*}
this shows that \eqref{pi4+1} holds when $\beta=0$.

Now suppose \eqref{pi4+1} holds for some $\beta\geq0$. Then,
\begin{align*}
   \pi\left(y^{(2k-1)}_{\beta+2,j}\right)&\geq \min_{i\geq1}\left\{\pi\left(y^{(2k-1)}_{\beta+1,i}\right)+\pi\left(m_{4i+1,i+j}\right)\right\}\\&
   \geq \min_{i\geq1}\left\{k+\left[\frac{7j-i-3}{4}\right]\right\}\\&
   \geq k+\left[\frac{7j-4}{4}\right],
\end{align*}
which is \eqref{pi4+1} with $\beta+1$ for $\beta$. This completes the proof.
\end{proof}

\begin{lemma}
For all $j,k\geq1$ and $\beta\geq 0$, we have
\begin{equation}\label{pi4}
   \pi\left(y^{(2k)}_{2\beta+1,j}\right)\geq k+\beta+\left[\frac{7j-6}{4}\right]
\end{equation}
   and
\begin{equation}\label{pi5}
\pi\left(y^{(2k)}_{2\beta+2,j}\right)\geq k+\beta+1+\left[\frac{7j-6}{4}\right].
\end{equation}
\end{lemma}
\begin{proof}
In view of Theorem \ref{T2}, we have
\begin{equation*}
        y_{1,j}^{(2k-1)}=x_{2k-1,j}.
\end{equation*}
From \eqref{pi1}, we can see that 
\begin{equation*}
\pi\left(x_{2k-1,j}\right)\geq k+\left[\frac{7j-4}{4}\right]\geq k+\left[\frac{7j-6}{4}\right],
\end{equation*}
this shows that \eqref{pi4} holds when $\beta=0$.

We now assume that \eqref{pi4} is true for some $\beta\geq 0$, then
\begin{align*}
   \pi\left(y^{(2k)}_{2\beta+2,j}\right)&\geq \min_{i\geq1}\left\{\pi\left(y^{(2k)}_{2\beta+1,i}\right)+\pi\left(m_{4i,i+j}\right)\right\}\\&
   \geq \min_{i\geq1}\left\{k+\beta+\left[\frac{7i-6}{4}\right]+\left[\frac{7j-i-1}{4}\right]\right\}\\&
   \geq k+\beta+\left[\frac{7j-2}{4}\right]\\&
   \geq k+\beta+1+\left[\frac{7j-6}{4}\right],
   \end{align*}
which is \eqref{pi5}. 

Now suppose \eqref{pi5} holds for some $\beta\geq0$. Then,
\begin{align*}
   \pi\left(y^{(2k)}_{2\beta+3,j}\right)&\geq \min_{i\geq1}\left\{\pi\left(y^{(2k)}_{2\beta+2,i}\right)+\pi\left(m_{4i+2,i+j}\right)\right\}\\&
   \geq \min_{i\geq1}\left\{k+\beta+1+\left[\frac{7i-6}{4}\right]+\left[\frac{7j-i-5}{4}\right]\right\}\\&
   \geq k+\beta+1+\left[\frac{7j-6}{4}\right],
\end{align*}
which is \eqref{pi4} with $\beta+1$ for $\beta$. This completes the proof.
\end{proof}
\noindent\textbf{Proof of Theorem \ref{th1}}
By \eqref{pi4+1} we see that $y_{2\beta+1, j}^{(2k-1)}\equiv 0 \pmod{7^{k}}$. So,  \eqref{c1} follows from \eqref{G1}.
Similarly, congruence \eqref{c2} follows from \eqref{G3} and \eqref{pi4}, and
congruence \eqref{c3} follows from \eqref{G4} together with \eqref{pi5}.

In view of \eqref{pi5}, we can express \eqref{G4} as
\begin{equation*}
\sum_{n\geq 0}p_{1,7^{2k}}\left(7^{2k+2\beta}n+\frac{22\cdot7^{2k+2\beta}+7^{2k}+1}{24}\right)q^n\equiv y^{(2k)}_{2\beta+2,1}\frac{f_7^{4}}{f_1^{6}}\pmod{7^{k+\beta+3}}.
\end{equation*}
Using \eqref{pi5} and the binomial theorem in the above equation, we obtain
\begin{equation}\label{B2}
\sum_{n\geq 0}p_{1,7^{2k}}\left(7^{2k+2\beta}n+\frac{22\cdot7^{2k+2\beta}+7^{2k}+1}{24}\right)q^n\equiv y^{(2k)}_{2\beta+2,1} f_7^{3}f_1\pmod{7^{k+\beta+2}}.
\end{equation}
From [\cite{BB}, p. 303, Entry 17(v)], we recall
that
\begin{equation}\label{B1}
 f_1=f_{49}\times\left(\dfrac{B(q^7)}{C(q^7)}-q\dfrac{A(q^7)}{B(q^7)}-q^2+q^5\dfrac{C(q^7)}{A(q^7)}\right),
\end{equation}
where 

$A(q):=\dfrac{f(-q^3;-q^4)}{f(-q^2)}$, $B(q):=\dfrac{f(-q^2;-q^5)}{f(-q^2)}$ and $C(q):=\dfrac{f(-q;-q^6)}{f(-q^2)}$.

Using \eqref{B1} in \eqref{B2} and then comparing the coefficients of $q^{7n+r}$ with $r\in \{3,4,6\}$, we arrive at \eqref{c4}.

% In view of \eqref{pi4+1}, we can express \eqref{G1} as
% \begin{equation*}
% \sum_{n\geq 0}p_{1,7^{2k-1}}\left(7^{2k+\beta-1}n+\frac{16\cdot7^{2k+\beta-1}+7^{2k-1}+1}{24}\right)q^n\equiv y^{(2k-1)}_{\beta+1,1}\frac{f_7^{3}}{f_1^{5}}\pmod{7^{k}}.
% \end{equation*}
% From \eqref{pi4+1} and the binomial theorem, we have
% \begin{equation}\label{B3}
% \sum_{n\geq 0}p_{1,7^{2k-1}}\left(7^{2k+\beta-1}n+\frac{16\cdot7^{2k+\beta-1}+7^{2k-1}+1}{24}\right)q^n\equiv y^{(2k-1)}_{\beta+1,1} f_7^{2}f_1^{2}\pmod{7^{k}}.
% \end{equation}
% Employing \eqref{B1} in \eqref{B3}, we obtain
% \begin{align}\label{B4}
% \sum_{n\geq 0}p_{1,7^{2k-1}}\left(7^{2k+\beta-1}n+\frac{16\cdot7^{2k+\beta-1}+7^{2k-1}+1}{24}\right)q^n & \equiv y^{(2k-1)}_{\beta+1,1} f_7^{2}f_{49}^{2}\left(\dfrac{B(q^7)}{C(q^7)} -q\dfrac{A(q^7)}{B(q^7)} \right. \nonumber \\ 
% & \left. -q^2+q^5\dfrac{C(q^7)}{A(q^7)}\right)^2\pmod{7^{k}}.
% \end{align}
% Extracting the terms containing $q^{7n+4}$ from both sides of the equation, we arrive at \eqref{c5}. 


\begin{thebibliography}{}
\bibitem{ABD}
Z. Ahmed, N.D. Baruah, and M.G. Dastidar, New congruences modulo 5 for the number of 2-color
partitions, J. Number Theory 157 (2015), 184--198.
\bibitem{BB}
B. C. Berndt, Ramanujan's Notebooks, Part III (Springer, New York, 1991).
\bibitem{SC}
S. Chern, New congruences for 2-color partitions, J. Number Theory 163 (2016), 474--481.
\bibitem{CS}
S. Chowla, Congruence properties of partitions, J. Lond. Math. Soc. 9, 247 (1934).
\bibitem{HH}
M. D. Hirschhorn and D. C. Hunt, A simple proof of the Ramanujan conjecture for powers of 5, J. Reine Angew. Math. 326 (1981), 1--17.
\bibitem{FG}
F.G. Garvan, A simple proof of Watson’s partition congruences for powers of 7, J. Aust. Math. Soc. 36, 316--334 (1984).
\bibitem{KGR}
K.G. Ramanathan, Identities and congruences of the Ramanujan type, Canad.J.Math., 2 (1950), 168--178.
\bibitem{Ram2}
S. Ramanujan: Collected Papers of Srinivasa Ramanujan. AMS, Chelsea (2000).
\bibitem{RD}
D. Ranganatha, Ramanujan-type congruences modulo powers of 5 and 7. Indian J. Pure Appl. Math. 48(2017), 449--465.
\bibitem{LW}
L. Wang, Congruences modulo powers of 5 for two restricted bipartitions. Ramanujan J., 44 (2017), 471--491.
\bibitem{GNW}
G. N. Watson, Ramanujans Vermutung \"{u}ber Zerf\"{a}llungsanzahlen, J. reine angew. Math. 179 (1938), 97--128.
\end{thebibliography}
\end{document}